\documentclass[a4paper,12pt]{article}

\usepackage{amsmath}
\usepackage{amsthm}
\usepackage{amssymb}
\usepackage{url}
\usepackage{mathdots}

\usepackage[latin1]{inputenc}
\usepackage[T1]{fontenc}

\usepackage{lmodern}

\usepackage[frenchb]{babel} 
\addto\captionsfrench{}

\usepackage[dvips=true,%
pdftitle={Un déterminant reproduisant},%
pdfauthor={Jean-François Burnol},%
pdfstartview=FitH,%
pdfpagemode=None%
]{hyperref}

\def\ladate{28 octobre 2010}

\usepackage[text={432pt,636pt}, hmarginratio=1:1, top=1in]{geometry}

\newtheorem*{theoreme}{Théorème}

\newtheorem*{proposition}{Proposition}

\newtheorem*{remarque}{Remarque}

\newcommand{\cX}{{\mathcal X}}
\newcommand{\cY}{{\mathcal Y}}
\newcommand{\cU}{{\mathcal U}}
\newcommand{\cV}{{\mathcal V}}

\DeclareMathOperator{\diag}{diag}

\begin{document}

\title{Un déterminant reproduisant}

\author{Jean-Fran\c cois Burnol}

\date{\ladate}

\maketitle

\begin{center}
\begin{small}

\begin{abstract}   Nous calculons le déterminant \[ D_{n+1} = \left| \frac{ y_ju_i -
       x_jv_i}{l_j - k_i} \right|_{1\leq i, j \leq n+1}\]
comme fonction du dernier des sextuplets $(u,v,k;x,y,l)$, en exprimant le
résultat sous une forme qui reproduit celle des entrées de $D_{n+1}$.
\end{abstract}

\begin{quote}
\bigskip
Universit\'e Lille 1\\ 
UFR de Math\'ematiques\\ 
Cit\'e scientifique M2\\ 
F-59655 Villeneuve d'Ascq\\ 
France\\
burnol@math.univ-lille1.fr\\
\end{quote}
\end{small}
\end{center}

\setlength{\normalbaselineskip}{16pt}
\baselineskip\normalbaselineskip
\setlength{\parskip}{6pt}

\section{Introduction}

Considérons un opérateur différentiel  de Sturm-Liouville sur un
intervalle  $I =\; ]a,b[$, de la forme $H(f) = -  (pf')' + qf$, et l'équation aux valeurs propres associée:
\begin{equation}
  \label{eq:1}
  H(f)(t) = - p (t) f''(t) - p'(t) f'(t) + q(t) f(t) = \lambda f(t)
\end{equation}
Si $g$ est un vecteur propre pour la valeur propre $\mu$, un
calcul familier donne:
\begin{equation}
  \label{eq:2}
  (\mu - \lambda) \int_{t_0}^{t_1} f(t)g(t)\,dt = \int_{t_0}^{t_1}
  \begin{vmatrix}
    f&- p f'' - p' f' + qf\\
g& - p g'' - p' g'+ qg 
  \end{vmatrix} \,dt = \left[- p \begin{vmatrix}
 f&f'\\g&g'
  \end{vmatrix}\right]_{t=t_0}^{t= t_1}
\end{equation}
Si nous supposons que $f$ et $g$ sont assujetties toutes deux à la même
condition initiale en $t = t_0$, nous obtenons la formule classique ($\mu\neq\lambda$):
\begin{equation}
  \label{eq:3}
  \int_{t_0}^{t_1} f(t)g(t)\,dt = p(t_1) \frac{g(t_1)f'(t_1) - g'(t_1)f(t_1)
    }{\mu - \lambda}
\end{equation}  Nous supposons ici pour simplifier que nos fonctions sont à
valeurs réelles, ainsi le terme de gauche peut être vu comme un produit
scalaire. Il est souvent utile (en particulier pour calculer des projections
orthogonales)  de savoir calculer le déterminant d'une matrice de  Gram $G =
(\int f_i(t)g_j(t)\,dt)_{1\leq i,j\leq n}$ formée avec ces produits
scalaires. On peut écrire le déterminant comme une intégrale multiple
(\cite[II, Nr 68]{polyaszego}), le résultat est du type
d'un produit scalaire formé par intégration sur un pavé $n$-dimensionnel, il ressemble plus au terme de gauche de l'équation
\eqref{eq:3} qu'à son terme de droite. 

Nous montrons qu'il existe d'autres expressions, qui elles sont du type du
terme de droite de l'équation \eqref{eq:3}. Quittant le domaine de l'Analyse,
il s'agit d'évaluer le déterminant
\begin{equation}
 \label{eq:4}
 D = \left| \dfrac{y_ju_i  - x_j v_i 
      }{l_j - k_i} \right|_{1\leq i, j \leq n+1}  
\end{equation}
dépendant de $6n + 6$ indéterminées $(u_i,v_i,k_i;
x_i,y_i,l_i)$, $1\leq i\leq n +1$. C'est
ce que nous faisons ici, en isolant le dernier sextuplet $(u,v,k;x,y,l)$ et
en montrant une identité
\begin{equation}
  \label{eq:5}
  \frac1d D = \dfrac{YU-XV}{l-k}
\end{equation} où $d$ est le mineur $n\times n$ principal de $D$ (ne dépendant
pas de $(u,v,k;x,y,l)$), et où $U$, $V$, $X$, et $Y$ sont des fonctions de
$(u,v,k;x,y,l)$ et des autres variables, mais avec la propriété que  $U$ et
$V$ ne dépendent pas de $(x,y,l)$ et que $X$ et $Y$ ne dépendent pas de
$(u,v,k)$.  

Les entrées $\dfrac{yu-xv}{l-k}$ de nos déterminants ont une forme que l'on
retrouve très souvent. Nous avons déjà cité  la théorie des opérateurs 
différentiels, on peut aussi évoquer par exemple les noyaux reproduisants
de la  théorie des polynômes orthogonaux (donnés par la formule de
Christoffel-Darboux \cite{simon, szego}).
L'identité \eqref{eq:5}
montre donc que la formation de déterminants en ces noyaux reproduisants donne
un résultat de la même forme caractéristique (dans certains contextes, il
peut être utile suivant les cas de faire des permutations du
style $U\leftrightarrow V$, $X\leftrightarrow -Y$, ou même
de modifier $U$, $V$, $X$, $Y$ par des facteurs imaginaires
ou même encore d'en prendre des combinaisons linéaires
appropriées).

On retrouve comme cas particulier (pour lequel
$(x_i,y_i,l_i) = (\overline{u_i},\overline{v_i},
\overline{k_i})$, $1\leq i\leq n$) le résultat obtenu dans
\cite[Thm. 2]{trivial}. Nous y posions la question de l'établir par
des calculs de déterminants, sans faire de récurrence sur leur
taille. En fait, il suffit pour cela de prendre pour point
de départ la formule (équivalente au Thm{.} 4.2 de Okada dans
\cite{okada}) que nous avons présentée dans
\cite[Thm. 1]{painleve} et de poursuivre le calcul avec l'aide
du cas particulier commun  aux identités de Jacobi (\cite{jacobi}, \cite[VI, \textsection
175]{muir}) et de Sylvester  (\cite{sylvester}, \cite[VI, \textsection
197]{muir}) et de quelques
observations auxiliaires. Nous incluons toutes les
démonstrations nécessaires afin d'éviter au lecteur non
familier de la théorie des déterminants d'avoir à
se reporter à d'autres sources.

\section{Énoncé du théorème principal}

Nous considérons $6n+6$ indéterminées $(u_i,v_i,k_i;x_i,y_i, l_i)$, $1\leq
i\leq n+1$, et posons $u = u_{n+1}$, $v = v_{n+1}$, $k = k_{n+1}$, $x =
x_{n+1}$, $y = y_{n+1}$, $l = l_{n+1}$. Nous notons $D_{n+1}(u,v,k;x,y,l)$ le déterminant de
taille $(n+1)\times(n+1)$:
\begin{equation}
  \label{eq:50} D_{n+1}(u,v,k;x,y,l) = \left| \dfrac{y_ju_i   -  x_jv_ i}{l_j - k_i}
\right|_{1\leq i, j \leq n+1}  
\end{equation}
Soit $D_n$ son mineur $n\times n$ principal situé en haut à
gauche. Posons:
\begin{equation}\label{eq:UX}
  U_n = \frac{1}{D_n} \begin{vmatrix}
      &\vdots& & \vdots\\
      \cdot\cdot&\dfrac{y_ju_i-x_jv_i}{l_j-k_i}&\cdot\cdot &u_i \\
      &\vdots& &\vdots\\
      \cdot\cdot&\dfrac{y_j u -x_j v}{l_j-k}&\cdot\cdot &u
    \end{vmatrix}_{n+1\times n+1}
\kern-.6cm  X_n = \frac{1}{D_n} \begin{vmatrix}
      &\vdots& &\vdots\\
      \cdot\cdot&\dfrac{y_ju_i-x_jv_i}{l_j-k_i}&\cdot\cdot&\dfrac{y u_i- x v_i}{l-k_i}\\
      &\vdots& &\vdots \\
\cdot\cdot&{x_j}&\cdot\cdot&x
    \end{vmatrix}_{n+1\times n+1}
\end{equation}
\begin{equation}
\label{eq:VY}
  V_n = \frac{1}{D_n} \begin{vmatrix}
      &\vdots& & \vdots\\
      \cdot\cdot&\dfrac{y_ju_i-x_jv_i}{l_j-k_i}&\cdot\cdot &v_i \\
      &\vdots& &\vdots\\
      \cdot\cdot&\dfrac{y_j u -x_j v}{l_j-k}&\cdot\cdot &v
     \end{vmatrix}_{n+1\times n+1}
  \kern-.6cm Y_n = \frac{1}{D_n} \begin{vmatrix}
      &\vdots& &\vdots\\
      \cdot\cdot&\dfrac{y_ju_i-x_jv_i}{l_j-k_i}&\cdot\cdot&\dfrac{y u_i- x v_i}{l-k_i}\\
       &\vdots& &\vdots\\
\cdot\cdot&{y_j}&\cdot\cdot&y
    \end{vmatrix}_{n+1\times n+1}
\end{equation}

\medskip
\begin{theoreme}
On a
\begin{equation}
  \frac{D_{n+1}(u,v,k;x,y,l)}{D_n} = \frac{Y_n(x,y,l)U_n(u,v,k) - X_n(x,y,l)V_n(u,v,k)}{l-k}
\end{equation}
\end{theoreme}

\section{Preuve}

Soit
\begin{equation}
  \label{eq:7}
  E = (-1)^{\frac{n(n+1)}2} \prod_{1\leq i,j \leq n+1} (l_j -
        k_i) \cdot \left| \frac{y_ju_i-x_jv_i}{l_j-k_i}\right|_{1\leq i,j\leq n+1} 
\end{equation}
où l'on rappelle que $u = u_{n+1}$, $v = v_{n+1}$, etc\dots

D'après \cite[Thm. 4.2]{okada} ou
\cite[Thm. 1]{painleve}, plus précisément sous sa forme
\cite[éq. (21)]{painleve}, 
\begin{equation}\label{eq:6}
E=   \begin{vmatrix}
    u_1 & u_2 & \cdot\cdot & u_n &u&      x_1 & x_2 & \cdot\cdot & x_n & x\\
    k_1 u_1 & k_2 u_2 & \cdot\cdot & k_n u_n & k u&      l_1 x_1 & l_2 x_2 & \cdot\cdot
    & l_n x_n & l x\\
\vdots & \vdots & \cdot\cdot & \vdots& \vdots& \vdots& \vdots &\cdot\cdot&\vdots&\vdots\\
    k_1^n u_1 & k_2^n u_2 & \cdot\cdot & k_n^n u_n & k^n u&      l_1^n x_1 & l_2^n
    x_2 & \cdot\cdot    & l_n^n x_n & l^n x\\
    v_1 & v_2 & \cdot\cdot & v_n &v&      y_1 & y_2 & \cdot\cdot & y_n & y\\
    k_1 v_1 & k_2 v_2 & \cdot\cdot & k_n v_n & k v&      l_1 y_1 & l_2 y_2 & \cdot\cdot
    & l_n y_n & l y\\
\vdots & \vdots & \cdot\cdot & \vdots& \vdots& \vdots& \vdots &\cdot\cdot&\vdots&\vdots\\
    k_1^n v_1 & k_2^n v_2 & \cdot\cdot & k_n^n v_n & k^n v&      l_1^n y_1 & l_2^n
    y_2 & \cdot\cdot    & l_n^n y_n & l^n y
  \end{vmatrix}
\end{equation}

\begin{proof}
  Il est simple d'aller de \eqref{eq:6} à
  \eqref{eq:7}. Soit $E'$ le déterminant défini dans
  \eqref{eq:6}. Soit $K$ (resp. $L$) la matrice dont la
  $i$\ieme{} ligne est $(k_1^{i-1}, \dots, k_n^{i-1}, k^{i-1})$
  (resp. $(l_1^{i-1}, \dots, l_n^{i-1}, l^{i-1})$, $1\leq
  i\leq n+1$), et soit
  $D_u$ la matrice $\diag(u_1,\dots,u_n,u)$, et sim. $D_v$,
  $D_x$, $D_y$. Alors
  \begin{equation}
    \begin{split}
      E' =
      \begin{vmatrix}
        KD_u & L D_x \\ KD_v & L D_y
      \end{vmatrix} = \begin{vmatrix} K D_u&0\\0&KD_v
      \end{vmatrix}\begin{vmatrix}
        I&D_u^{-1} K^{-1} L D_x\\I&D_v^{-1} K^{-1} L D_y
      \end{vmatrix} \\= \det(K)^2 \det(D_u K^{-1} L D_y - D_v
      K^{-1} L D_x)
    \end{split}
  \end{equation}
Soit $C_1(t)$, \dots, $C_{n+1}(t)$ des polynômes de degrés
au plus $n$, et $C$ la matrice ayant leurs coordonnées dans
$(1,t,\dots,t^n)$ en lignes. Alors $C K= (C_i(k_j))$ et
$C L = (C_i(l_j))$. Posons $A(t) = \prod_{1\leq i\leq n+1}
(t - k_i)$ et prenons $C_i(t) = A(t)/(t-k_i)$. On obtient
$K^{-1} = \diag(A'(k_i)^{-1}) C$ donc $K^{-1} L =
\diag(A'(k_i)^{-1})(C_i(l_j))$ et
\begin{equation}
  K^{-1} L = \diag_{1\leq i \leq n+1}(A'(k_i)^{-1})(\frac1{l_j-k_i})\diag_{1\leq j \leq n+1}(A(l_j))
\end{equation}
Donc
\begin{align}
  E' &= \det(K)^2\prod_i A'(k_i)^{-1}\prod_j
  A(l_j)\det(\frac{u_i y_j - v_i x_j}{l_j-k_i})_{1\leq
    i,j\leq n+1}\\
&= (-1)^{\frac{n(n+1)}2} \prod_{1\leq i,j\leq n+1} (l_j-k_i)\det(\frac{ y_ju_i -  x_jv_i}{l_j-k_i})_{1\leq
    i,j\leq n+1}
\end{align}
Ainsi effectivement $E = E'$.
\end{proof}

On va utiliser un cas particulier commun aux identités de Jacobi
(voir \cite[VI, \textsection175]{muir}) et de Sylvester (\cite{sylvester}, \cite[VI, \textsection197]{muir}).
\begin{proposition}[Jacobi, \cite{jacobi}]
    Soit $n>2$, et $M = (m_{ij})$ une matrice $n\times
  n$. Notons $E = \det M$ et soit $D$ le mineur obtenu dans $E$
  en supprimant les deux premières lignes et les deux
  premières colonnes. Alors
  \begin{equation}
    ED =
    \begin{vmatrix}
      M_{11}&M_{12}\\M_{21}&M_{22}
    \end{vmatrix}
  \end{equation}
avec $M_{ij}$ le mineur obtenu de $E$ en supprimant la
$i$\ieme{} ligne et la $j$\ieme{} colonne. Une identité semblable vaut pour deux
lignes quelconques et deux colonnes quelconques.
\end{proposition}
\begin{proof}[Preuve de l'identité de Jacobi]
  Il suffit de faire la preuve avec les entrées de $M$ des
  indéterminées indépendantes. Il existe  une (unique) combinaison
  $\lambda_3 C_3 + \dots +\lambda_n C_n$
   des colonnes de numéros allant de  $3$ à $n$ qui, soustraite à la
  première colonne $C_1$, annule tous ses éléments sauf les
  deux du haut. De même il existe une  combinaison
  $\mu_3 C_3 + \dots +\mu_n C_n$ qui, soustraite à la
   colonne $C_2$ annule tous ses éléments sauf les deux
   premiers. Si l'on transforme le déterminant $E$ par ces
   combinaisons, on a alors une forme par bloc et
   \begin{equation}
     E =
     \begin{vmatrix}m_{11} - \sum_{3\leq j\leq n} \lambda_j
       m_{1j}&\qquad 
       m_{12} - \sum_{3\leq j\leq n} \mu_j m_{1j}\\m_{21} -
       \sum_{3\leq j\leq n}\lambda_j m_{2j}&\qquad 
       m_{22} - \sum_{3\leq j\leq n} \mu_j m_{2j}
     \end{vmatrix} \cdot D
   \end{equation}
Or dans le calcul du mineur $M_{11}$ par exemple, la même
soustraction (une fois seulement) de colonnes donne
immédiatement $  M_{11} = (m_{22} - \sum_{3\leq j\leq n} \mu_j m_{2j})\;D
$. Idem pour les trois autres mineurs et on trouve donc
\begin{equation}
  E = \begin{vmatrix}\frac{\vphantom{|_|}M_{22}}D&\frac{\vphantom{|_|}M_{21}}D\\\noalign{\medskip} \frac{\vphantom{|_|}M_{12}}D&\frac{\vphantom{|_|}M_{11}}D
     \end{vmatrix} \cdot D = \frac{M_{11} M_{22} - M_{12}
       M_{21}} D
\end{equation}
ce qui prouve le résultat.
\end{proof} 

\begin{remarque} Cette proposition est le cas $k=2$ de l'énoncé plus général
suivant, qu'on démontrerait à l'identique: soit $E =
|m_{ij}|$ de taille $n\times n$, soit $k<n$, soit $F$ le mineur
$(n-k)\times(n-k)$ diagonal inférieur droit dans $E$, soit  $G =
|g_{ij}|_{1\leq i,j\leq k}$ le déterminant $k\times k$  avec $g_{ij}$  le
mineur $(n-k+1)\times(n-k+1)$ de $E$ obtenu en bordant $F$ par $m_{ij}$ (et
$m_{iq}$, $k+1\leq q\leq n$, $m_{qj}$, $k+1\leq q\leq n$). Alors $G = F^{k-1}
E$. Sylvester \cite{sylvester} a un énoncé encore plus général en bordant par
plusieurs lignes et  colonnes. Et en ce qui concerne Jacobi, il
s'agit du calcul des mineurs de l'adjoint (formé avec les co-facteurs) de $E$. Considérons $K =
|\widetilde{m}_{ij}|_{1\leq i,j\leq k}$ le mineur principal $k\times k$ de
l'adjoint $\widetilde E$ ($k\geq2$). En appliquant ce qui vient d'être montré
à $E$ privé de la $i$\ieme{} ligne et de la $j$\ieme{} colonne, on a
$\widetilde{g}_{ij} = F^{k-2} \widetilde{m}_{ij}$. Donc $K = F^{k(2-k)}
\widetilde G = F^{k(2-k)} G^{k-1} = F E^{k-1}$. Ceci est la formule de Jacobi,
qui calcule les mineurs de $\widetilde E$. Pour des énoncés synthétiques avec
des lignes et colonnes arbitraires, voir l'article Déterminants, section F, de
\cite{edm2}. Notre proposition est F(3), Jacobi est F(2) et (le cas 
particulier vu ici de) Sylvester est F(4).
\end{remarque}

Nous appliquons ceci au déterminant de l'équation \eqref{eq:6}, en considérant
les quatre entrées $k^nu$, $k^n v$, $l^nx$, $l^n y$ dont nous notons les
co-mineurs, respectivement $\cY$, $\cX$, $\cV$, et $\cU$. À la limite pour
$l\to\infty$ on a
$ l^{-n} E = (-1)^{n+1} x \cV + y \cU + O(\frac1l)$. Or:
\begin{equation}
  \label{eq:14}
    E = (-1)^{\frac{n(n+1)}2} \prod_{1\leq i,j \leq
      n} (l_j - k_i) (l-k) \prod_{1\leq i\leq n} (l-k_i) \prod_{1\leq j\leq n}
    (l_j -k)
    \begin{vmatrix}
      &\vdots& & \vdots\\
      \cdots&\frac{y_ju_i-x_jv_i}{l_j-k_i}&\cdots &\frac{y u_i-x v_i}{l-k_i}\\
      &\vdots& &\vdots\\
    \end{vmatrix}_{1\leq i,j\leq n+1}
\end{equation}
\begin{equation}
  \label{eq:15}
  \lim_{l \to \infty}  l^{-n} E = (-1)^{\frac{n(n+1)}2} \prod_{1\leq i,j \leq
      n} (l_j - k_i) \prod_{1\leq j\leq n}
    (l_j -k)
    \begin{vmatrix}
      &\vdots& & \vdots\\
      \cdots&\frac{y_ju_i-x_jv_i}{l_j-k_i}&\cdots &{y u_i-x v_i}\\
      &\vdots& &\vdots\\
    \end{vmatrix}_{1\leq i,j\leq n+1}
\end{equation}
\begin{equation}
  \cV = (-1)^{\frac{n(n-1)}2} \prod_{1\leq i,j \leq
      n} (l_j - k_i) \prod_{1\leq j\leq n}
    (l_j -k)    \begin{vmatrix}
      &\vdots& & \vdots\\
      \cdots&\frac{y_ju_i-x_jv_i}{l_j-k_i}&\cdots &v_i \\
      &\vdots& &\vdots\\
    \end{vmatrix}_{1\leq i,j\leq n+1}
\end{equation}
\begin{equation}
  \label{eq:16}
  \cU = (-1)^{\frac{n(n+1)}2} \prod_{1\leq i,j \leq
      n} (l_j - k_i) \prod_{1\leq j\leq n}
    (l_j -k)    \begin{vmatrix}
      &\vdots& & \vdots\\
      \cdots&\frac{y_ju_i-x_jv_i}{l_j-k_i}&\cdots &u_i \\
      &\vdots& &\vdots\\
    \end{vmatrix}_{1\leq i,j\leq n+1}
\end{equation}

À la limite pour
$k\to\infty$ on a
$ k^{-n} E = u \cY + (-1)^{n+1} v \cX + O(\frac1k)$. Or:
\begin{equation}
    E = (-1)^{\frac{n(n+1)}2} \prod_{1\leq i,j \leq
      n} (l_j - k_i) (l-k) \prod_{1\leq i\leq n} (l-k_i) \prod_{1\leq j\leq n}
    (l_j -k)
    \begin{vmatrix}
      &\vdots& \\
      \cdots&\frac{y_ju_i-x_jv_i}{l_j-k_i}&\cdots\\
      &\vdots& \\
\cdots&\frac{y_ju-x_jv}{l_j-k}&\cdots
    \end{vmatrix}_{1\leq i,j\leq n+1}
\end{equation}
\begin{equation}
  \lim_{k \to \infty}  k^{-n} E = (-1)^{\frac{n(n-1)}2} \prod_{1\leq i,j \leq
      n} (l_j - k_i) \prod_{1\leq i\leq n}
    (l -k_i)
    \begin{vmatrix}
      &\vdots& \\
      \cdots&\frac{y_ju_i-x_jv_i}{l_j-k_i}&\cdots\\
      &\vdots& \\
\cdots&{y_ju-x_jv}&\cdots
    \end{vmatrix}_{1\leq i,j\leq n+1}
\end{equation}
\begin{equation}
  \cY =  (-1)^{\frac{n(n-1)}2} \prod_{1\leq i,j \leq
      n} (l_j - k_i) \prod_{1\leq i\leq n}
    (l -k_i)
    \begin{vmatrix}
      &\vdots& \\
      \cdots&\frac{y_ju_i-x_jv_i}{l_j-k_i}&\cdots\\
      &\vdots& \\
\cdots&{y_j}&\cdots
    \end{vmatrix}_{1\leq i,j\leq n+1}
\end{equation}
\begin{equation}
  \label{eq:27}
  \cX = (-1)^{\frac{n(n+1)}2} \prod_{1\leq i,j \leq
      n} (l_j - k_i) \prod_{1\leq j\leq n}
    (l -k_i)
    \begin{vmatrix}
      &\vdots& \\
      \cdots&\frac{y_ju_i-x_jv_i}{l_j-k_i}&\cdots\\
      &\vdots& \\
\cdots&{x_j}&\cdots
    \end{vmatrix}_{1\leq i,j\leq n+1}
\end{equation}
Avec les quantités $\cU$, $\cV$, $\cX$, et $\cY$ ainsi définies on obtient:
\begin{equation}
  \label{eq:30}
  ED = {\cY\cU -\cX\cV}
\end{equation}
\begin{equation}
  \label{eq:31}
  D = (-1)^{\frac{n(n-1)}2} \prod_{1\leq i,j \leq n} (l_j -
        k_i) \cdot \left| \frac{y_ju_i-x_jv_i}{l_j-k_i}\right|_{1\leq i,j\leq n} 
\end{equation}
Avec $U_n$, $V_n$, $X_n$, $Y_n$, $D_n$ définis par \eqref{eq:UX},  
\eqref{eq:VY}, 
\begin{equation}
  {\cU}= (-1)^{\frac{n(n+1)}2} {\prod_{1\leq i,j \leq
      n} (l_j - k_i) \prod_{1\leq j\leq n}
    (l_j -k)  }D_n U_n
  \end{equation}
\begin{equation}
{\cV} = (-1)^{\frac{n(n-1)}2} {\prod_{1\leq i,j \leq
      n} (l_j - k_i) \prod_{1\leq j\leq n}
    (l_j -k)  }D_n V_n
  \end{equation}
\begin{equation}
  {\cX} = (-1)^{\frac{n(n+1)}2}  {\prod_{1\leq i,j \leq
      n} (l_j - k_i) \prod_{1\leq i\leq n}
    (l -k_i)  } D_n X_n 
 \end{equation}
\begin{equation}
{\cY}  = (-1)^{\frac{n(n-1)}2} {\prod_{1\leq i,j \leq
      n} (l_j - k_i) \prod_{1\leq i\leq n}
    (l -k_i)  } D_n Y_n
 \end{equation}
On obtient finalement
\begin{equation}
  (-1)^n \left| \frac{y_ju_i-x_jv_i}{l_j-k_i}\right|_{1\leq i,j\leq n+1}
  \left| \frac{y_ju_i-x_jv_i}{l_j-k_i}\right|_{1\leq i,j\leq n} = (-1)^n D_n^2
  \frac{Y_n U_n - X_n V_n}{l-k}
\end{equation}
ce qui prouve le Théorème.

Au passage, nous avons obtenu par cette preuve des représentations de $U_n$, $V_n$, $X_n$
et $Y_n$ à l'aide de déterminants de taille $(2n+1)\times (2n+1)$:
\begin{equation}\label{eq:Ubig}
  U_n = \frac{\begin{vmatrix}
    u_1 & u_2 & \cdot\cdot & u_n &u&      x_1 & x_2 & \cdot\cdot & x_n \\
    k_1 u_1 & k_2 u_2 & \cdot\cdot & k_n u_n & k u&      l_1 x_1 & l_2 x_2 & \cdot\cdot
    & l_n x_n \\
\vdots & \vdots & \cdot\cdot & \vdots& \vdots& \vdots& \vdots &\cdot\cdot&\vdots\\
    k_1^n u_1 & k_2^n u_2 & \cdot\cdot & k_n^n u_n & k^n u&      l_1^n x_1 & l_2^n
    x_2 & \cdot\cdot    & l_n^n x_n \\
    v_1 & v_2 & \cdot\cdot & v_n &v&      y_1 & y_2 & \cdot\cdot & y_n \\
    k_1 v_1 & k_2 v_2 & \cdot\cdot & k_n v_n & k v&      l_1 y_1 & l_2 y_2 & \cdot\cdot
    & l_n y_n \\
\vdots & \vdots & \cdot\cdot & \vdots& \vdots& \vdots& \vdots &\cdot\cdot&\vdots\\
    k_1^{n-1} v_1 & k_2^{n-1} v_2 & \cdot\cdot & k_n^{n-1} v_n & k^{n-1} v&      l_1^{n-1} y_1 & l_2^{n-1}
    y_2 &\cdot\cdot    & l_n^{n-1} y_n 
  \end{vmatrix}} {(-1)^n(-1)^{\frac{n(n-1)}2} \prod_{1\leq i,j \leq
      n} (l_j - k_i)D_n \;  \prod_{1\leq j\leq n}
    (l_j -k)  }
\end{equation}
\begin{equation}\label{eq:Vbig}
V_n = \frac{\begin{vmatrix}
    u_1 & u_2 & \cdot\cdot & u_n &u&      x_1 & x_2 & \cdot\cdot & x_n \\
    k_1 u_1 & k_2 u_2 & \cdot\cdot & k_n u_n & k u&      l_1 x_1 & l_2 x_2 & \cdot\cdot
    & l_n x_n \\
\vdots & \vdots & \cdot\cdot & \vdots& \vdots& \vdots& \vdots &\cdot\cdot&\vdots\\
    k_1^{n-1} u_1 & k_2^{n-1} u_2 & \cdot\cdot & k_n^{n-1} u_n & k^{n-1} u&      l_1^{n-1} x_1 & l_2^{n-1}
    x_2 & \cdot\cdot    & l_n^{n-1} x_n \\
    v_1 & v_2 & \cdot\cdot & v_n &v&      y_1 & y_2 & \cdot\cdot & y_n \\
    k_1 v_1 & k_2 v_2 & \cdot\cdot & k_n v_n & k v&      l_1 y_1 & l_2 y_2 & \cdot\cdot
    & l_n y_n \\
\vdots & \vdots & \cdot\cdot & \vdots& \vdots& \vdots& \vdots &\cdot\cdot&\vdots\\
    k_1^{n} v_1 & k_2^{n} v_2 & \cdot\cdot & k_n^{n} v_n & k^{n} v&      l_1^{n} y_1 & l_2^{n}
    y_2 & \cdot\cdot    & l_n^{n} y_n 
  \end{vmatrix}} {(-1)^{\frac{n(n-1)}2} \prod_{1\leq i,j \leq
      n} (l_j - k_i)  D_n \; \prod_{1\leq j\leq n}
    (l_j -k)  }
  \end{equation}
\begin{equation}
 X_n = \frac{\begin{vmatrix}
    u_1 & u_2 & \cdot\cdot & u_n &      x_1 & x_2 & \cdot\cdot & x_n & x\\
    k_1 u_1 & k_2 u_2 & \cdot\cdot & k_n u_n &      l_1 x_1 & l_2 x_2 & \cdot\cdot
    & l_n x_n & l x\\
\vdots & \vdots & \cdot\cdot & \vdots& \vdots& \vdots &\cdot\cdot&\vdots&\vdots\\
    k_1^{n} u_1 & k_2^{n} u_2 & \cdot\cdot & k_n^{n} u_n &       l_1^{n} x_1 & l_2^{n}
    x_2 & \cdot\cdot    & l_n^{n} x_n & l^{n} x\\
    v_1 & v_2 & \cdot\cdot & v_n &      y_1 & y_2 & \cdot\cdot & y_n & y\\
    k_1 v_1 & k_2 v_2 & \cdot\cdot & k_n v_n &      l_1 y_1 & l_2 y_2 & \cdot\cdot
    & l_n y_n & l y\\
\vdots & \vdots & \cdot\cdot & \vdots&  \vdots& \vdots &\cdot\cdot&\vdots&\vdots\\
    k_1^{n-1} v_1 & k_2^{n-1} v_2 & \cdot\cdot & k_n^{n-1} v_n &       l_1^{n-1} y_1 & l_2^{n-1}
    y_2 & \cdot\cdot    & l_n^{n-1} y_n & l^{n-1} y
  \end{vmatrix}}  {(-1)^n(-1)^{\frac{n(n-1)}2} \prod_{1\leq i,j \leq
      n} (l_j - k_i)  D_n \; \prod_{1\leq i\leq n}
    (l -k_i)  } 
 \end{equation}
\begin{equation}
Y_n  = \frac{\begin{vmatrix}
    u_1 & u_2 & \cdot\cdot & u_n &      x_1 & x_2 & \cdot\cdot & x_n & x\\
    k_1 u_1 & k_2 u_2 & \cdot\cdot & k_n u_n &      l_1 x_1 & l_2 x_2 & \cdot\cdot
    & l_n x_n & l x\\
\vdots & \vdots & \cdot\cdot & \vdots& \vdots& \vdots &\cdot\cdot&\vdots&\vdots\\
    k_1^{n-1} u_1 & k_2^{n-1} u_2 & \cdot\cdot & k_n^{n-1} u_n &      l_1^{n-1} x_1 & l_2^{n-1}
    x_2 & \cdot\cdot    & l_n^{n-1} x_n & l^{n-1} x\\
    v_1 & v_2 & \cdot\cdot & v_n &      y_1 & y_2 & \cdot\cdot & y_n & y\\
    k_1 v_1 & k_2 v_2 & \cdot\cdot & k_n v_n &      l_1 y_1 & l_2 y_2 & \cdot\cdot
    & l_n y_n & l y\\
\vdots & \vdots & \cdot\cdot & \vdots& \vdots& \vdots &\cdot\cdot&\vdots&\vdots\\
    k_1^{n} v_1 & k_2^{n} v_2 & \cdot\cdot & k_n^{n} v_n &      l_1^{n} y_1 & l_2^{n}
    y_2 & \cdot\cdot    & l_n^{n} y_n & l^{n} y
  \end{vmatrix}} {(-1)^{\frac{n(n-1)}2} \prod_{1\leq i,j \leq
      n} (l_j - k_i) D_n \; \prod_{1\leq i\leq n}
    (l -k_i)  } 
 \end{equation}
Dans chacune de ces expressions apparaît au dénominateur la
quantité 
\begin{equation}
  D = (-1)^{\frac{n(n-1)}2} \prod_{1\leq i,j \leq
      n} (l_j - k_i) \left| \frac{y_ju_i-x_jv_i}{l_j-k_i}\right|_{1\leq i,j\leq n} 
\end{equation}
qui est à chaque fois le mineur $2n\times 2n$ obtenu en
supprimant dans le dénominateur la ligne et la colonne
contenant, respectivement, $k^n u$, $k^n v$, $l^n x$, $l^n y$.

\section{Un déterminant factorisant}

Nous examinons le cas particulier:
\begin{equation}
  x_i = u_i, y_i = - v_i, l_i = -k_i, 1\leq i\leq n
\end{equation}
\begin{equation}
 1\leq i,j\leq n\implies \frac{y_j u_i - x_j v_i}{l_j - k_i} = \frac{u_i v_j  + v_j u_i}{k_i + k_j}
\end{equation}
Les formules \eqref{eq:UX} et \eqref{eq:VY} montrent qu'en considérant $U_n$,
$V_n$, $X_n$, $Y_n$ comme fonctions de trois variables, les $n$ triplets
$(u_i,v_i,k_i)$, $1\leq i\leq n$ étant fixés, on a les relations:
\begin{align}
  X_n(u,-v,-k) &= U_n(u,v,k)\\
Y_n(u,-v,-k) &= -V_n(u,v,k)
\end{align}
\begin{equation}
\frac{D_{n+1}(u,v,k;x,y,l)}{D_n} =
\frac{U_n(u,v,k)V_n(x,-y,-l)+V_n(u,v,k)U_n(x,-y,-l)}{k + (-l)}
\end{equation}
Si nous imposons les relations supplémentaires $x=u$, $y=-v$, $l=-k$, nous
obtenons la factorisation
\begin{equation}\label{eq:DD}
  \frac{D_{n+1}(u,v,k)}{D_n} = \frac1k U_n(u,v,k) V_n(u,v,k)
\end{equation}
L'équation \eqref{eq:Ubig} donne comme  valeur de $U_n$:
\begin{equation}\label{eq:Ubig2}
  \frac{\begin{vmatrix}
    u_1 & u_2 & \cdot\cdot & u_n &u&      u_1 & u_2 & \cdot\cdot & u_n \\
    k_1 u_1 & k_2 u_2 & \cdot\cdot & k_n u_n & k u&      -k_1 u_1 & -k_2 u_2 & \cdot\cdot
    & -k_n u_n \\
\vdots & \vdots & \cdot\cdot & \vdots& \vdots& \vdots& \vdots &\cdot\cdot&\vdots\\
    k_1^n u_1 & k_2^n u_2 & \cdot\cdot & k_n^n u_n & k^n u&      (-1)^n k_1^n
    u_1 & (-1)^n k_2^n
    u_2 & \cdot\cdot    &(-1)^n k_n^n u_n \\
    v_1 & v_2 & \cdot\cdot & v_n &v&      -v_1 & -v_2 & \cdot\cdot & -v_n \\
    k_1 v_1 & k_2 v_2 & \cdot\cdot & k_n v_n & k v&      k_1 v_1 & k_2 v_2 & \cdot\cdot
    & k_n v_n \\
\vdots & \vdots & \cdot\cdot & \vdots& \vdots& \vdots& \vdots &\cdot\cdot&\vdots\\
    k_1^{n-1} v_1 & k_2^{n-1} v_2 & \cdot\cdot & k_n^{n-1} v_n & k^{n-1} v&
    (-1)^n k_1^{n-1} v_1 & (-1)^n k_2^{n-1}
    v_2 &\cdot\cdot    & (-1)^n k_n^{n-1} v_n 
  \end{vmatrix}} {(-1)^{\frac{n(n+1)}2} D_n \prod_{1\leq i,j \leq
      n} (k_i + k_j) \prod_{1\leq j\leq n}
    (k+k_j)  }
\end{equation}
On fait les $\left[\frac{n}2\right]$ permutations de lignes
$2k\leftrightarrow 2k+n+1$, $1\leq 2k\leq n$. Puis, lorsque $n$ est impair il
faut aussi mettre la $(n+1)$\ieme{} ligne en dernière position. On a donc un
signe $(-1)^n (-1)^{\frac{n(n-1)}2} = (-1)^{\frac{n(n+1)}2}$. Puis, on soustrait les
$n$ premières colonnes des $n$ dernières. Ceci met le déterminant sous une
forme par blocs, et donne, si $n$ est pair:
\begin{equation}
  U_n = \frac{\begin{vmatrix}
    u_1 & u_2 & \cdot\cdot & u_n &u\\
    k_1 v_1 & k_2 v_2 & \cdot\cdot & k_n v_n & k v\\
\vdots & \vdots & \cdot\cdot & \vdots& \vdots\\
    k_1^n u_1 & k_2^n u_2 & \cdot\cdot & k_n^n u_n & k^n u
  \end{vmatrix}_{n+1\times n+1}\begin{vmatrix}
    v_1 & v_2 & \cdot\cdot & v_n \\
    k_1 u_1 & k_2 u_2 & \cdot\cdot & k_n u_n \\
\vdots & \vdots & \cdot\cdot & \vdots\\
    k_1^{n-1} u_1 & k_2^{n-1} u_2 & \cdot\cdot & k_n^{n-1} u_n 
  \end{vmatrix}_{n\times n}}{2^{-n}D_n\prod_{1\leq i,j \leq
      n} (k_i + k_j) \prod_{1\leq j\leq n}
    (k+k_j)}
\end{equation}
et pour $n$ impair:
\begin{equation}
  U_n = \frac{\begin{vmatrix}
    u_1 & u_2 & \cdot\cdot & u_n \\
    k_1 v_1 & k_2 v_2 & \cdot\cdot & k_n v_n \\
\vdots & \vdots & \cdot\cdot & \vdots\\
    k_1^{n-1} u_1 & k_2^{n-1} u_2 & \cdot\cdot & k_n^{n-1} u_n 
  \end{vmatrix}_{n\times n}\begin{vmatrix}
    v_1 & v_2 & \cdot\cdot & v_n & v\\
    k_1 u_1 & k_2 u_2 & \cdot\cdot & k_n u_n & ku\\
\vdots & \vdots & \cdot\cdot & \vdots&\vdots\\
    k_1^{n} u_1 & k_2^{n} u_2 & \cdot\cdot & k_n^{n} u_n & k^n u
  \end{vmatrix}_{n+1\times n+1}}{ 2^{-n} D_n\prod_{1\leq i,j \leq
      n} (k_i + k_j) \prod_{1\leq j\leq n}
    (k+k_j)}
\end{equation}

L'équation \eqref{eq:Vbig} donne comme valeur de $V_n$:
\begin{equation}\label{eq:Vbig2}
  \frac{\begin{vmatrix}
    u_1 & u_2 & \cdot\cdot & u_n &u&      u_1 & u_2 & \cdot\cdot & u_n \\
    k_1 u_1 & k_2 u_2 & \cdot\cdot & k_n u_n & k u&      -k_1 u_1 & -k_2 u_2 & \cdot\cdot
    & -k_n u_n \\
\vdots & \vdots & \cdot\cdot & \vdots& \vdots& \vdots& \vdots &\cdot\cdot&\vdots\\
    k_1^{n-1} u_1 & k_2^{n-1} u_2 & \cdot\cdot & k_n^{n-1} u_n & k^{n-1} u&      (-1)^{n-1} k_1^{n-1}
    u_1 & (-1)^{n-1} k_2^{n-1}
    u_2 & \cdot\cdot    &(-1)^{n-1} k_n^{n-1} u_n \\
    v_1 & v_2 & \cdot\cdot & v_n &v&      -v_1 & -v_2 & \cdot\cdot & -v_n \\
    k_1 v_1 & k_2 v_2 & \cdot\cdot & k_n v_n & k v&      k_1 v_1 & k_2 v_2 & \cdot\cdot
    & k_n v_n \\
\vdots & \vdots & \cdot\cdot & \vdots& \vdots& \vdots& \vdots &\cdot\cdot&\vdots\\
    k_1^{n} v_1 & k_2^{n} v_2 & \cdot\cdot & k_n^{n} v_n & k^{n} v&
    (-1)^{n+1} k_1^{n} v_1 & (-1)^{n+1} k_2^{n}
    v_2 &\cdot\cdot    & (-1)^{n+1} k_n^{n} v_n 
  \end{vmatrix}} {(-1)^{\frac{n(n-1)}2} D_n \prod_{1\leq i,j \leq
      n} (k_i + k_j) \prod_{1\leq j\leq n}
    (k+k_j)  }
\end{equation}
On fait les $\left[\frac{n}2\right]$ permutations de lignes
$2k\leftrightarrow 2k+n$, $1\leq 2k\leq n$. Puis, lorsque $n$ est impair il
faut aussi mettre la dernière ligne en $(n+1)$\ieme{} position. On obtient, pour $n$  pair:
\begin{equation}
  V_n = \frac{\begin{vmatrix}
    u_1 & u_2 & \cdot\cdot & u_n \\
    k_1 v_1 & k_2 v_2 & \cdot\cdot & k_n v_n \\
\vdots & \vdots & \cdot\cdot & \vdots\\
    k_1^{n-1} v_1 & k_2^{n-1} v_2 & \cdot\cdot & k_n^{n-1} v_n 
  \end{vmatrix}_{n\times n}\begin{vmatrix}
    v_1 & v_2 & \cdot\cdot & v_n & v\\
    k_1 u_1 & k_2 u_2 & \cdot\cdot & k_n u_n& ku \\
\vdots & \vdots & \cdot\cdot & \vdots&\vdots\\
    k_1^{n} v_1 & k_2^{n} v_2 & \cdot\cdot & k_n^{n} v_n & k^{n} v
  \end{vmatrix}_{n+1\times n+1}}{2^{-n}D_n\prod_{1\leq i,j \leq
      n} (k_i + k_j) \prod_{1\leq j\leq n}
    (k+k_j)}
\end{equation}
et pour $n$ impair:
\begin{equation}
  V_n = \frac{\begin{vmatrix}
    u_1 & u_2 & \cdot\cdot & u_n & u\\
    k_1 v_1 & k_2 v_2 & \cdot\cdot & k_n v_n & kv\\
\vdots & \vdots & \cdot\cdot & \vdots & \vdots\\
    k_1^{n} v_1 & k_2^n v_2 & \cdot\cdot & k_n^{n} v_n & k^n v
  \end{vmatrix}_{n+1\times n+1}\begin{vmatrix}
    v_1 & v_2 & \cdot\cdot & v_n \\
    k_1 u_1 & k_2 u_2 & \cdot\cdot & k_n u_n \\
\vdots & \vdots & \cdot\cdot & \vdots\\
    k_1^{n-1} v_1 & k_2^{n-1} v_2 & \cdot\cdot & k_n^{n-1} v_n 
  \end{vmatrix}_{n\times n}}{2^{-n} D_n\prod_{1\leq i,j \leq
      n} (k_i + k_j) \prod_{1\leq j\leq n}
    (k+k_j)}
\end{equation}

La même technique, en partant de
\eqref{eq:6}, permet la factorisation de $D_{n+1}$ (et de
$D_{n}$). On obtient:
\begin{equation}
  D_{n+1} = 2^{n+1}\frac{\begin{vmatrix}
    u_1 & u_2 & \cdot\cdot & u_n & u \\
    k_1 v_1 & k_2 v_2 & \cdot\cdot & k_n v_n & kv\\
    k_1^2 u_1 & k_2^2 u_2 & \cdot\cdot & k_n^2 u_n & k^2 u\\
\vdots & \vdots & \cdot\cdot & \vdots&\vdots\\
    & & \cdots & 
  \end{vmatrix}_{n+1\times n+1} \begin{vmatrix}
    v_1 & v_2 & \cdot\cdot & v_n & v\\
    k_1 u_1 & k_2 u_2 & \cdot\cdot & k_n u_n  & ku \\
    k_1^2 v_1 & k_2^2 v_2 & \cdot\cdot & k_n^2 v_n & k^2 v\\
\vdots & \vdots & \cdot\cdot & \vdots&\vdots\\
    & & \cdots &  
  \end{vmatrix}_{n+1\times n+1}}{\prod_{1\leq i,j \leq n+1} (k_i+
        k_j) }
\end{equation}
\begin{equation}
   D_n = \left|\frac{u_i v_j  + v_j u_i}{k_i + k_j}\right|_{n\times n} = 2^n\frac{\begin{vmatrix}
    u_1 & u_2 & \cdot\cdot & u_n \\
    k_1 v_1 & k_2 v_2 & \cdot\cdot & k_n v_n \\
\vdots & \vdots & \cdot\cdot & \vdots\\
    & & \cdots & 
  \end{vmatrix}_{n\times n} \begin{vmatrix}
    v_1 & v_2 & \cdot\cdot & v_n \\
    k_1 u_1 & k_2 u_2 & \cdot\cdot & k_n u_n \\
\vdots & \vdots & \cdot\cdot & \vdots\\
    & & \cdots &  
  \end{vmatrix}_{n\times n}}{\prod_{1\leq i,j \leq n} (k_i+
        k_j) }
\end{equation}
Donc le déterminant
\begin{equation}\label{eq:fin1}
  \begin{vmatrix}
    u_1 & u_2 & \cdot\cdot & u_n &u\\
    k_1 v_1 & k_2 v_2 & \cdot\cdot & k_n v_n & k v\\
    k_1^2 u_1 & k_2^2 u_2 & \cdot\cdot & k_n^2 u_n & k^2 u\\
\vdots & \vdots & \cdot\cdot & \vdots&\vdots\\
 & \cdots & \cdots & \cdots & 
  \end{vmatrix}_{n+1\times n+1}
\end{equation}
divisé par son mineur principal de taille $n\times n$ et par le produit
$\prod_{1\leq j\leq n} (k+k_j)$ est égal, pour $n$ pair à $U_n$ (dans ce cas
il y a $k^n u$ en position $(n+1,n+1)$ dans \eqref{eq:fin1}) et pour $n$
impair à $V_n$ (dans ce cas
il y a $k^n v$ en position $(n+1,n+1)$). Le déterminant
\begin{equation}\label{eq:fin2}
  \begin{vmatrix}
    v_1 & v_2 & \cdot\cdot & v_n & v\\
    k_1 u_1 & k_2 u_2 & \cdot\cdot & k_n u_n& ku \\
    k_1^2 v_1 & k_2^2 v_2 & \cdot\cdot & k_n^2 v_n & k^2 v\\
\vdots & \vdots & \cdot\cdot & \vdots&\vdots\\
 & \cdots & \cdots & \cdots & 
  \end{vmatrix}_{n+1\times n+1}
\end{equation}
divisé par son mineur principal de taille $n\times n$ et par le produit
$\prod_{1\leq j\leq n} (k+k_j)$ est égal à $V_n$ pour $n$ pair (il y a alors
$k^n v$ en bas à droite dans \eqref{eq:fin2}) et à $U_n$
pour $n$ impair ($k^n u$ en bas à droite  dans \eqref{eq:fin2}).
Conformément à \eqref{eq:DD}, le quotient $\frac{D_{n+1}}{D_n}$ est égal à
$\frac1k U_n V_n$  (pour $x=u$, $y=-v$, $l=-k$).

\end{document}